\documentclass[12pt,oneside,english]{amsart}
\usepackage[T1]{fontenc}
\usepackage[latin9]{inputenc}
\usepackage{geometry}
\usepackage{color}
\usepackage{babel}
\usepackage{amstext}
\usepackage{amsthm}
\usepackage{amssymb}
\usepackage[unicode=true,
 bookmarks=true,bookmarksnumbered=false,bookmarksopen=false,
 breaklinks=false,pdfborder={0 0 0},pdfborderstyle={},backref=false,colorlinks=true]
 {hyperref}
\hypersetup{
 backref}

\makeatletter
\numberwithin{equation}{section}
\numberwithin{figure}{section}
\theoremstyle{plain}
\newtheorem{thm}{\protect\theoremname}[section]
\theoremstyle{plain}
\newtheorem{cor}[thm]{\protect\corollaryname}
\theoremstyle{plain}
\newtheorem{lem}[thm]{\protect\lemmaname}
\theoremstyle{plain}
\newtheorem{prop}[thm]{\protect\propositionname}
\theoremstyle{remark}
\newtheorem{rem}[thm]{\protect\remarkname}
\theoremstyle{remark}
\newtheorem{claim}[thm]{\protect\claimname}
\theoremstyle{definition}
\newtheorem{defn}[thm]{\protect\definitionname}

\DeclareMathOperator{\supp}{supp}
\allowdisplaybreaks

\makeatother

\providecommand{\claimname}{Claim}
\providecommand{\corollaryname}{Corollary}
\providecommand{\definitionname}{Definition}
\providecommand{\lemmaname}{Lemma}
\providecommand{\propositionname}{Proposition}
\providecommand{\remarkname}{Remark}
\providecommand{\theoremname}{Theorem}

\begin{document}
\title{A note on one sided extrapolation of compactness and applications}
\begin{abstract}
In this paper one sided counterparts of compactness extrapolation
results of Hytönen and Lappas \cite{HL} are provided. As a consequence
of those results, compactness results for one sided singular integrals,
commutators of one sided fractional integrals and a certain class
of $L^{r}$-Hörmander operators are provided. The results for commutators
of $L^{r}$-Hörmander operators seem new even in the classical setting.
\end{abstract}

\author{Francisco J. Martín-Reyes}
\address{(Francisco J. Martín-Reyes) Departamento de Análisis Matemático, Estadística
e Investigación Operativa y Matemática Aplicada. Facultad de Ciencias.
Universidad de Málaga (Málaga, Spain).}
\email{martin\_reyes@uma.es}
\author{Israel P. Rivera-Ríos}
\address{(Israel P. Rivera Ríos) Departamento de Análisis Matemático, Estadística
e Investigación Operativa y Matemática Aplicada. Facultad de Ciencias.
Universidad de Málaga (Málaga, Spain). }
\email{israelpriverarios@uma.es}
\thanks{Both authors were partially supported by Junta de Andalucía FQM 354
and by Ministerio de Ciencia e Innovación, Spain, grant PID2022-136619NB-I00.}
\subjclass[2000]{42B20, 42B25.}
\maketitle

\section{Introduction and main results}

We recall that an operator $T:X\rightarrow Y$ between two Banach
spaces is compact if $T(B_{X})$ (where $B_{X}$ is the closed unit
ball in $X$) has compact closure in $Y$.

Within the framework of singular integrals and related operators,
probably the first result on compact operators was obtained by Uchiyama
\cite{U}, who showed that if $T$ is a homogeneous singular integral
with smooth kernel and $b\in CMO$, namely, $b$ is a function in
the closure of $C_{c}^{\infty}(\mathbb{R}^{n})$ in the $BMO$ norm,
then, the commutator 
\[
[b,T]=b(x)Tf(x)-T(bf)(x)
\]
is a compact operator on $L^{p}(\mathbb{R}^{n})$ for every $1<p<\infty$. 

Since Uchiyama's paper, a number of authors have devoted a number
of works to obtain results on compactness of several operators, such
as Calderón-Zygmund operators \cite{V,PPV,V2,MS,OP}, or commutators
of $CMO$ symbols with several operators (see papers such as \cite{unW1,unW2,unW3,unW4,unW5,unW6}).

In the weighted setting, motivated by the study of the Beltrami equation,
Clop and Cruz \cite{CC} provided a counterpart of Uchiyama's result.
We recall that $w$ is a weight if it is a positive and finite a.e.
function and that $w\in A_{p}$ for $p>1$ if 
\[
\sup_{Q}\left(\frac{1}{|Q|}\int_{Q}w\right)\left(\frac{1}{|Q|}\int_{Q}w^{-\frac{1}{p-1}}\right)^{p-1}<\infty.
\]
As we were saying, Clop and Cruz showed that if $T$ is a Calderón-Zygmund
operator and $b\in CMO$ then for every $1<p<\infty$ and every $w\in A_{p}$
we have that $[b,T]$ is compact on $L^{p}(w)$. 

In the last years, compactness in the weighted setting has been an
active field of research, see for instance papers such as \cite{rec1,rec2,GWY,rec4,rec5}.
Even more recently some results related to compactness of commutators
in the two weight setting have appeared \cite{LL,MM,MMW}.

Our purpose in this paper is to provide results of compactness in
the one sided weighted setting. As far as we know, just one paper
has been devoted to this topic. García and Ortega \cite{GO} showed
that if $T$ is a one sided Calderón-Zygmund operator with kernel
supported in $(-\infty,0)$ and $b\in CMO$ then $[b,T]$ is compact
on $L^{p}(w)$ for $w\in A_{p}^{+}$ (see the definition of $A_{p}^{+}$
a few lines below). 

Our main results in this paper are related to extrapolation of the
compactness in the one sided setting. Before presenting them let us
briefly discuss the state of the art in the classical setting. Very
recently Hytönen and Lappas \cite{HL} settled an interesting result.
We remit the interested reader as well to \cite{HL2} for a partial
extension to the bilinear setting and \cite{COY} for a full extension
to the multilinear setting and a number of applications. Hytönen and
Lapas' result says that assuming that a linear or sublinear operator
$T$ is bounded on $L^{p}(w_{0})$ for some $p_{0}\in(\lambda,\infty)$
and every $w_{0}\in A_{p_{0}/\lambda}$ where $\lambda\geq1$ (with
operator norm dominated by an increasing function in $[w_{0}]_{A_{p_{0/\lambda}}}$)
which is compact on $L^{p_{1}}(w_{1})$ for some $p_{1}\in(\lambda,\infty)$
and just some $w_{1}\in A_{p_{1}/\lambda}$ we have that $T$ is compact
on $L^{p}(w)$ for every $p\in(\lambda,\infty)$ and every $w\in A_{p/\lambda}$. 

Our first main theorem is a one sided version of that result. Before
presenting the statement we recall that $w\in A_{p}^{+}$ for $1<p<\infty$
if 
\[
[w]_{A_{p}^{+}}=\sup_{a<b<c}\frac{1}{c-a}\int_{a}^{b}w\left(\frac{1}{c-a}\int_{b}^{c}w^{-\frac{1}{p-1}}\right)^{p-1}<\infty
\]
and analogously that $w\in A_{p}^{-}$ if 
\[
[w]_{A_{p}^{-}}=\sup_{a<b<c}\frac{1}{c-a}\int_{b}^{c}w\left(\frac{1}{c-a}\int_{a}^{b}w^{-\frac{1}{p-1}}\right)^{p-1}<\infty.
\]

\begin{thm}
\label{thm:CompactExtrapolationAp}Let $\lambda\geq1$ and $p_{0},p_{1}\in(\lambda,\infty)$.
Let $T$ be a linear or a sublinear bounded on $L^{p_{0}}(w_{0})$
for every $w_{0}\in A_{p_{0}/\lambda}^{+}$ with the operator norm
dominated by some increasing function of $[w_{0}]_{A_{p_{0/\lambda}}^{+}}$.
Suppose in addition that $T$ is compact on $L^{p_{1}}(w_{1})$ for
some $w_{1}\in A_{p_{1}/\lambda}^{-}$. Then $T$ is compact on $L^{p}(w)$
for all $p\in(\lambda,\infty)$ and all $w\in A_{p/\lambda}^{+}$.
\end{thm}

In the one-sided setting, the natural hypothesis for $w_{1}$ that
one may expect to work is that $w_{1}\in A_{p_{1}/\lambda}^{+}$.
However, our approach leads to $w_{1}\in A_{p_{1}/\lambda}^{-}$.
In Remark \ref{rem:HyAp+-} we will provide a further insight on this.
All in all, since $A_{p_{1}/\lambda}\subset A_{p_{1}/\lambda}^{-}$
we have the following corollary.
\begin{cor}
\label{thm:CorAp}Let $\lambda\geq1$ and $p_{0},p_{1}\in(\lambda,\infty)$.
Let $T$ be a linear or a sublinear bounded on $L^{p_{0}}(w_{0})$
for every $w_{0}\in A_{p_{0}/\lambda}^{+}$ with the operator norm
dominated by some increasing function of $[w_{0}]_{A_{p_{0/\lambda}}^{+}}$.
Suppose in addition that $T$ is compact on $L^{p_{1}}(w_{1})$ for
some $w_{1}\in A_{p_{1}/\lambda}$. Then $T$ is compact on $L^{p}(w)$
for all $p\in(\lambda,\infty)$ and all $w\in A_{p/\lambda}^{+}$.
\end{cor}

A particular case of interest that is quite useful in applications
is that $w_{1}\equiv1$, namely that $T$ is compact in the unweighted
setting. We present that particular case in the next Corollary.
\begin{cor}
\label{thm:CorUnWeighted}Let $\lambda\geq1$ and $p_{0},p_{1}\in(\lambda,\infty)$.
Let $T$ be a linear or a sublinear bounded on $L^{p_{0}}(w_{0})$
for every $w_{0}\in A_{p_{0}/\lambda}^{+}$ with the operator norm
dominated by some increasing function of $[w_{0}]_{A_{p_{0/\lambda}}^{+}}$.
Suppose in addition that $T$ is compact on $L^{p_{1}}$. Then $T$
is compact on $L^{p}(w)$ for all $p\in(\lambda,\infty)$ and all
$w\in A_{p/\lambda}^{+}$.
\end{cor}

Our next result is a one sided counterpart of the corresponding off-diagonal
result in \cite{HL}. First we recall that given $1<p\leq q<\infty$,
$w\in A_{p,q}^{+}$ if 
\[
[w]_{A_{p,q}^{+}}=\sup_{a<b<c}\left(\frac{1}{c-a}\int_{a}^{b}w^{q}\right)^{\frac{1}{q}}\left(\frac{1}{c-a}\int_{b}^{c}w^{-p'}\right)^{\frac{1}{p'}}<\infty
\]
and analogously, $w\in A_{p,q}^{-}$ if
\[
[w]_{A_{p,q}^{-}}=\sup_{a<b<c}\left(\frac{1}{c-a}\int_{b}^{c}w^{q}\right)^{\frac{1}{q}}\left(\frac{1}{c-a}\int_{a}^{b}w^{-p'}\right)^{\frac{1}{p'}}<\infty.
\]

\begin{thm}
\label{thm:CompactExtrapolationApq}Let $T$ be a linear operator
defined and bounded from $L^{p_{0}}(w_{0}^{p_{0}})$ to $L^{q_{0}}(w_{0}^{q_{0}})$
for some $1<p_{0}\leq q_{0}<\infty$ and every $w_{0}\in A_{p_{0},q_{0}}^{+}$
with the operator norm dominated by some increasing function of $[w_{0}]_{A_{p_{0,}q_{0}}^{+}}$.
Assume additionally that $T$ is compact from $L^{p_{1}}(w_{1}^{p_{1}})$
to $L^{q_{1}}(w_{1}^{q_{1}})$ with for some $w_{1}\in A_{p_{1},q_{1}}^{-}$
where $1<p_{1}\leq q_{1}<\infty$ with $\frac{1}{p_{0}}-\frac{1}{q}_{0}=\frac{1}{p_{1}}-\frac{1}{q_{1}}$.
Then $T$ is compact from $L^{p}(w^{p})$ to $L^{q}(w^{q})$ for every
$1<p\leq q<\infty$ with $\frac{1}{p_{0}}-\frac{1}{q_{0}}=\frac{1}{p}-\frac{1}{q}$
and all $w\in A_{p,q}^{+}$.
\end{thm}

As in the case above, since $A_{p_{1},q_{1}}\subset A_{p_{1},q_{1}}^{-}$
we have the following Corollary.
\begin{cor}
\label{thm:CorApq}Let $T$ be a linear operator defined and bounded
from $L^{p_{0}}(w_{0}^{p_{0}})$ to $L^{q_{0}}(w^{q_{0}})$ for some
$1<p_{0}\leq q_{0}<\infty$ and every $w_{0}\in A_{p_{0}}^{+}$ with
the operator norm dominated by some increasing function of $[w_{0}]_{A_{p_{0,}q_{0}}^{+}}$.
Assume additionally that $T$ is compact from $L^{p_{1}}(w_{1}^{p_{1}})$
to $L^{q_{1}}(w_{0}^{q_{1}})$ for some $w_{1}\in A_{p_{1},q_{1}}$
where $1<p_{1}\leq q_{1}<\infty$ with $\frac{1}{p_{0}}-\frac{1}{q}_{0}=\frac{1}{p_{1}}-\frac{1}{q_{1}}$.
Then $T$ is compact from $L^{p}(w^{p})$ to $L^{q}(w^{q})$ for every
$1<p\leq q<\infty$ with $\frac{1}{p_{0}}-\frac{1}{q}_{0}=\frac{1}{p}-\frac{1}{q}$
and all $w\in A_{p,q}^{+}$.
\end{cor}

Since in most situations the most interesting choice for $w_{1}$
is $w_{1}\equiv1$, we record that case in the following Corollary.
\begin{cor}
\label{thm:CorUnWeightedApq}Let $T$ be a linear operator defined
and bounded from $L^{p_{0}}(w_{0}^{p_{0}})$ to $L^{q_{0}}(w^{q_{0}})$
for some $1<p_{0}\leq q_{0}<\infty$ and every $w_{0}\in A_{p_{0},q_{0}}^{+}$
with the operator norm dominated by some increasing function of $[w_{0}]_{A_{p_{0,}q_{0}}^{+}}$.
Assume additionally that $T$ is compact from $L^{p_{1}}$ to $L^{q_{1}}$
where $1<p_{1}\leq q_{1}<\infty$ with $\frac{1}{p_{0}}-\frac{1}{q}_{0}=\frac{1}{p_{1}}-\frac{1}{q_{1}}$.
Then $T$ is compact from $L^{p}(w^{p})$ to $L^{q}(w^{q})$ for every
$1<p\leq q<\infty$ with $\frac{1}{p_{0}}-\frac{1}{q}_{0}=\frac{1}{p}-\frac{1}{q}$
and all $w\in A_{p,q}^{+}$.
\end{cor}

The remainder of the paper is organized as follows. Section \ref{sec:ProofsMR}
is devoted to the proofs of the main results and in Section \ref{sec:Applications}
we gather the applications of these main results. There we provide
results on commutators, and singular integrals. In particular we provide
a result for certain $L^{r}$-Hörmander singular integrals that seems
to be new even in the classical setting, and is related to \cite{unW5}.

\section{Proofs of the main results\label{sec:ProofsMR} }

In this section we settle Theorems \ref{thm:CompactExtrapolationAp}
and \ref{thm:CompactExtrapolationApq}. We will follow the strategy
in \cite{HL}. In order to provide our proofs we need some preliminaries
that are contained in the following subsection.

\subsection{Preliminaries}

\subsubsection{Extrapolation and interpolation results}

We begin recalling two extrapolation results that will be needed in
the proof of the main theorems. The first of them can be obtained
from minor modifications of \cite[Theorem 3.25]{CUMPBook}
\begin{thm}
\label{thm:Extrapolation}Let $\lambda\geq1$ and $p_{0}\in(\lambda,\infty)$.
Let $T$ be a linear or a sublinear bounded on $L^{p_{0}}(w_{0})$
for every $w\in A_{p_{0}/\lambda}^{+}$ with operator norm bounded
by some increasing function of $[w]_{A_{p_{0}/\lambda}^{+}}$. Then
$T$ is bounded on $L^{p}(w)$ for all $p\in(\lambda,\infty)$ and
every $w\in A_{p/\lambda}^{+}$.
\end{thm}

The other extrapolation result we will rely upon is the following.
\begin{thm}[\cite{MR}]
\label{thm:Extrapolationpq}Let $T$ be a linear operator defined
and bounded from $L^{p_{0}}(w_{0}^{p_{0}})$ to $L^{q_{0}}(w^{q_{0}})$
for some $1<p_{0}\leq q_{0}<\infty$ and every $w_{0}\in A_{p_{0},q_{0}}^{+}$
with operator norm bounded by some increasing function of $[w]_{A_{p_{0},q_{0}}^{+}}$.
Then $T$ is bounded from $L^{p}(w^{p})$ to $L^{q}(w^{q})$ for every
$1<p\leq q<\infty$ such that $\frac{1}{p_{0}}-\frac{1}{q}_{0}=\frac{1}{p}-\frac{1}{q}$,
and all $w\in A_{p,q}^{+}$.
\end{thm}

Another fundamental tool for our purposes is the following abstract
theorem due to Cwikel and Kalton \cite{CK}. Here we state the version
contained in \cite{HL}.
\begin{thm}
\label{thm:CwKa} Let $(X_{0},X_{1})$ and $(Y_{0},Y_{1})$ be Banach
couples and let $T$ be a linear operator such that $T:X_{0}+X_{1}\to Y_{0}+Y_{1}$
and $T:X_{j}\to Y_{j}$ boundedly for $j=0,1$. Suppose moreover that
$T:X_{1}\to Y_{1}$ is compact. Let $[\ ,\ ]_{\theta}$ be the complex
interpolation functor of Calderón. Then also $T:[X_{0},X_{1}]_{\theta}\to[Y_{0},Y_{1}]_{\theta}$
is compact for $\theta\in(0,1)$ under any of the following four side
conditions: 
\begin{enumerate}
\item $X_{1}$ has the UMD (unconditional martingale differences) property, 
\item $X_{1}$ is reflexive, and $X_{1}=[X_{0},E]_{\alpha}$ for some Banach
space $E$ and $\alpha\in(0,1)$, 
\item $Y_{1}=[Y_{0},F]_{\beta}$ for some Banach space $F$ and $\beta\in(0,1)$, 
\item $X_{0}$ and $X_{1}$ are both complexified Banach lattices of measurable
functions on a common measure space. 
\end{enumerate}
\end{thm}

For the UMD property we remit the interested reader to \cite[Chapter 4]{HvNVW}.

In our case, since we are dealing with weighted $L^{p}$ spaces the
following Lemma will serve our purposes.
\begin{lem}
\label{lem:Cond4} If $p_{j}\in[1,\infty)$ and $w_{j}$ are weights,
then the spaces $X_{j}=L^{p_{j}}(w_{j})$ satisfy the condition 4
of Theorem \ref{thm:CwKa}. 
\end{lem}

\begin{proof}
Clearly, $X_{j}=L^{p_{j}}(w_{j})$, $j=0,1$, are complexified Banach
lattices of measurable functions on the common measure space $\mathbb{R}$
and hence we are done.
\end{proof}

\subsubsection{Facts on one-sided $A_{p}$ and $A_{p,q}$ weights}

We need some facts on one-sided $A_{p}$ weights. Our first result
is the following reverse Hölder inequality which was obtained in \cite{CUNO}.
However we shall use the version contained in \cite{MRdT}. Before
that we recall that $A_{\infty}^{+}=\bigcup_{p\geq1}A_{p}^{+}$ and
that $A_{\infty}^{-}=\bigcup_{p\geq1}A_{p}^{-}$.
\begin{lem}
\label{lem:LemRHI}Assume that $w\in A_{\infty}^{+}$ then, there
exists $r=r_{w}>1$ such that
\[
\left(\frac{1}{b-a}\int_{a}^{b}w^{r}\right)^{\frac{1}{r}}\leq4\frac{1}{c-b}\int_{b}^{c}w
\]
where $a<c$ and $b=\frac{c+a}{2}$. Analogously, if $w\in A_{\infty}^{-}$
then there exists $s=s_{w}>1$ such that
\[
\left(\frac{1}{c-b}\int_{b}^{c}w^{s}\right)^{\frac{1}{s}}\leq4\frac{1}{b-a}\int_{a}^{b}w.
\]
\end{lem}

The following result shows that it suffices to settle $A_{p}$-like
conditions with a gap in order to show that the same condition holds
without a gap.
\begin{lem}
\label{lem:Gap}Let $p,q>1$ and let $t>2$. Assume that $v^{q}$
and $u^{-p}$ are no negative locally integrable functions. If for
every interval $I=(a,b)$ we have that
\begin{equation}
\left(\frac{1}{\frac{l_{I}}{t}}\int_{a}^{a+\frac{l_{I}}{t}}v^{q}\right)^{\frac{1}{q}}\left(\frac{1}{\frac{l_{I}}{t}}\int_{b-\frac{l_{I}}{t}}^{b}u^{-p'}\right)^{\frac{1}{p'}}\leq K,\label{eq:gapCondition}
\end{equation}
where $l_{I}=b-a$, then
\[
\sup_{a<b<c}\left(\frac{v^{q}(a,b)}{c-a}\right)^{\frac{1}{q}}\left(\frac{u^{-p'}(b,c)}{c-a}\right)^{\frac{1}{p'}}\leq K.
\]
\end{lem}

\begin{proof}
It suffices to show that for every interval $(a,c)$
\[
\left(\frac{v^{q}(a,b)}{c-a}\right)^{\frac{1}{q}}\left(\frac{u^{-p'}(b,c)}{c-a}\right)^{\frac{1}{p'}}\leq K.
\]
Let us fix an interval $(a,c)$. Let $x\in(a,c)$. We split $(a,c)$
as follows. We define $x_{0}=a$, $x_{k+1}-x_{k}=\frac{1}{t}(x-x_{k})$.
Then we have that 
\[
(a,x)=\bigcup_{k=0}^{\infty}(x_{k},x_{k+1}].
\]
Now we fix $k$ and let $y<x$ such that $x-y=x_{k+1}-x_{k}$. Hence,
by (\ref{eq:gapCondition}),
\begin{align*}
v^{q}(x_{k},x_{k+1}) & \leq K^{q}(x_{k+1}-x_{k})\left(\frac{x-y}{u^{-p'}(y,x)}\right)^{\frac{q}{p'}}\\
 & \leq K^{q}(x_{k+1}-x_{k})\left(M_{u^{-p'}}^{-}\left(u^{p'}\chi_{(a,x)}\right)(x)\right)^{\frac{q}{p'}}
\end{align*}
where 
\[
M_{\rho}^{-}f(x)=\sup_{h>0}\frac{1}{\rho(x-h,x)}\int_{x-h}^{x}f(y)\rho(y)dy.
\]
Summing in $k$, 
\[
v^{q}(a,x)\leq K^{q}(x-a)\left(M_{u^{-p'}}^{-}\left(u^{p'}\chi_{(a,x)}\right)(x)\right)^{\frac{q}{p'}}
\]
and consequently
\[
\frac{v^{q}(a,x)}{x-a}\leq K^{q}\left(M_{u^{-p'}}^{-}\left(u^{p'}\chi_{(a,x)}\right)(x)\right)^{\frac{q}{p'}}.
\]
Now let $b\in(a,c)$. For every $x\in(b,c)$ we have that 
\begin{align*}
\frac{v^{q}(a,b)}{c-a}\leq\frac{v^{q}(a,x)}{x-a} & \leq K^{q}\left(M_{u^{-p'}}^{-}\left(u^{p'}\chi_{(a,x)}\right)(x)\right)^{\frac{q}{p'}}\\
 & \leq K^{q}\left(M_{u^{-p'}}^{-}\left(u^{p'}\chi_{(a,c)}\right)(x)\right)^{\frac{q}{p'}}.
\end{align*}
Note, that since 
\[
(b,c)\subset\left\{ x\,:\,M_{u^{-p'}}^{-}\left(u^{p'}\chi_{(a,c)}\right)(x)\geq\left(\frac{v^{q}(a,b)}{(c-a)K^{q}}\right)^{\frac{p'}{q}}\right\} ,
\]
by the weak type $(1,1)$ of $M_{u^{-p'}}$ (see \cite{B}) we have
that 
\begin{align*}
u^{-p'}(b,c) & \leq u^{-p'}\left(\left\{ x\,:\,M_{u^{-p'}}^{-}\left(u^{p'}\chi_{(a,c)}\right)(x)\geq\left(\frac{v^{q}(a,b)}{(c-a)K^{q}}\right)^{\frac{p'}{q}}\right\} \right)\\
 & \leq\left(\frac{K^{q}(c-a)}{v^{q}(a,b)}\right)^{\frac{p'}{q}}\int_{\mathbb{R}}u^{p'}\chi_{(a,c)}u^{-p'}=\left(\frac{K^{q}(c-a)}{v^{q}(a,b)}\right)^{\frac{p'}{q}}(c-a),
\end{align*}
from what readily follows that
\[
\left(\frac{v^{q}(a,b)}{c-a}\right)^{\frac{1}{q}}\left(\frac{u^{-p'}(b,c)}{c-a}\right)^{\frac{1}{p'}}\leq K
\]
and hence we are done.
\end{proof}
\begin{prop}
\label{prop:Apq}Let $1<p\leq q<\infty$.
\begin{itemize}
\item $w\in A_{p,q}^{+}$, iff $w^{q}\in A_{1+\frac{q}{p'}}^{+}$ iff $w^{-p'}\in A_{1+\frac{p'}{q}}^{-}$.
\item $w\in A_{p,q}^{-}$ iff $w^{q}\in A_{1+\frac{q}{p'}}^{-}$ iff $w^{-p'}\in A_{1+\frac{p'}{q}}^{+}.$
\end{itemize}
\end{prop}

\subsection{Proofs of the main theorems}

Finally we are in the position to settle our main theorems. We begin
with the proof of Theorem \ref{thm:CompactExtrapolationAp} which
relies upon the following Lemma.
\begin{lem}
\label{lem:Interp}Let $\lambda\in[1,\infty)$, let $q,q_{1}\in(\lambda,\infty)$
and $v\in A_{q/\lambda}^{+}$, $v_{1}\in A_{q_{1}/\lambda}^{-}$.
Then 
\[
L^{q}(v)=[L^{q_{0}}(v_{0}),L^{q_{1}}(v_{1})]_{\gamma}
\]
for some $q_{0}\in(\lambda,\infty)$, $v_{0}\in A_{q_{0}/\lambda}^{+}$,
and $\gamma\in(0,1)$.
\end{lem}

Assuming that we have already settled this Lemma, we are in the position
to provide our proof of Theorem \ref{thm:CompactExtrapolationAp}. 
\begin{proof}[Proof of Theorem \ref{thm:CompactExtrapolationAp}]
By Theorem \ref{thm:Extrapolation}, since $T$ is a bounded linear
operator on $L^{p_{0}}(w)$ for some $p_{0}\in(\lambda,\infty)$ and
every $w\in A_{p_{0}/\lambda}^{+},$ then it is bounded as well for
on $L^{p}(w)$ for all $p\in(\lambda,\infty)$ and all $w\in A_{p/\lambda}^{+}$.
Additionally, we assumed that $T$ is a compact operator on $L^{p_{1}}(w_{1})$
for some $p_{1}\in(\lambda,\infty)$ and some $w_{1}\in A_{p_{1}/\lambda}^{-}$.
We need to show that $T$ is compact on $L^{p}(w)$ for all $p\in(\lambda,\infty)$
and all $w\in A_{p/\lambda}^{+}$. Now, fix some $p\in(\lambda,\infty)$
and $w\in A_{p/\lambda}^{+}$. By Lemma \ref{lem:Interp}, we have
that
\[
L^{p}(w)=[L^{p_{0}}(w_{0}),L^{p_{1}}(w_{1})]_{\theta}
\]
for some $p_{0}\in(\lambda,\infty)$, some $w_{0}\in A_{p_{0}/\lambda}^{+}$,
and some $\theta\in(0,1)$. Choosing $X_{j}=Y_{j}=L^{p_{j}}(w_{j})$,
we know that $T:X_{0}+X_{1}\to Y_{0}+Y_{1}$, that $T:X_{j}\to Y_{j}$
is bounded, since $T$ is bounded on all $L^{q}(w)$ with $q\in(\lambda,\infty)$
and $w\in A_{q/\lambda}^{+}$ as we noted above, and that $T:X_{1}\to Y_{1}$
is compact by assumption. Lemma \ref{lem:Cond4}, assures that the
last condition 4 of Theorem \ref{thm:CwKa} is also satisfied by the
spaces $X_{j}=L^{p_{j}}(w_{j})$. Hence, by Theorem \ref{thm:CwKa},
it follows that $T$ is also compact on $[X_{0},X_{1}]_{\theta}=[Y_{0},Y_{1}]_{\theta}=L^{p}(w)$. 
\end{proof}
For Theorem \ref{thm:CompactExtrapolationApq} the argument is analogous.
We rely upon the following Lemma.
\begin{lem}
\label{lem:Interppq}Let $1<p_{1}\leq q_{1}<\infty$, $1<p\leq q<\infty$,
$w_{1}\in A_{p_{1},q_{1}}^{-}$, $w\in A_{p,q}^{+}$. Then 
\[
L^{p}(w^{p})=[L^{p_{0}}(w_{0}),L^{p_{1}}(w_{1})]_{\gamma}\qquad L^{q}(w^{q})=[L^{q_{0}}(w_{0}),L^{q_{1}}(w_{1})]_{\gamma}
\]
for some $1<p_{0}\leq q_{0}<\infty$, $w_{0}\in A_{p_{0},q_{0}}^{+}$,
and $\gamma\in(0,1)$.
\end{lem}

Again, if we assume that we have already settled that Lemma, we are
in the position to settle Theorem \ref{thm:CompactExtrapolationApq}.
\begin{proof}[Proof of Theorem \ref{thm:CompactExtrapolationApq}]
By Theorem \ref{thm:CompactExtrapolationApq}, $L^{p_{0}}(w_{0}^{p_{0}})$
to $L^{q_{0}}(w^{q_{0}})$ for some $1<p_{0}\leq q_{0}<\infty$ and
every $w_{0}\in A_{p_{0},q_{0}}^{+}$, then it is bounded as well
from $L^{p}(w^{p})$ to $L^{q}(w^{q})$ for every $1<p\leq q<\infty$
and all $w\in A_{p,q}^{+}$. We have additionally that $T$ is a compact
operator from $L^{p_{1}}(w_{1}^{p_{1}})$ to $L^{q_{1}}(w_{1}^{q_{1}})$
and some $w_{1}\in A_{p_{1},q_{1}}^{-}.$ We need to show that $T$
is compact from $L^{p}(w^{p})$ to $L^{q}(w^{q})$ for all $1<p\leq q<\infty$
and every $w\in A_{p,q}^{+}$. Now, fix some $1<p\leq q<\infty$ and
$w\in A_{p,q}^{+}$. By Lemma \ref{lem:Interppq}, we have that
\[
L^{p}(w^{p})=[L^{p_{0}}(w_{0}),L^{p_{1}}(w_{1})]_{\theta}\qquad L^{q}(w^{q})=[L^{q_{0}}(w_{0}),L^{q_{1}}(w_{1})]_{\theta}
\]
for some $1<p_{0}\leq q_{0}<\infty$ and some $w_{0}\in A_{p_{0},q_{0}}^{+}$,
and some $\theta\in(0,1)$. Choosing $X_{j}=L^{p_{j}}(w_{j})$, $Y_{j}=L^{q_{j}}(w_{j})$,
we know that $T:X_{0}+X_{1}\to Y_{0}+Y_{1}$, that $T:X_{j}\to Y_{j}$
is bounded, since $T$ is bounded from $L^{p}(w^{p})$ to $L^{q}(w^{q})$
for every $1<p\leq q<\infty$ and every $w\in A_{p,q}^{+}$ as we
noted above, and that $T:X_{1}\to Y_{1}$ is compact by assumption.
Lemma \ref{lem:Cond4}, assures that the last condition 4 of Theorem
\ref{thm:CwKa} also hold for $X_{j}=L^{p_{j}}(w_{j}^{p_{j}})$. Hence,
by Theorem \ref{thm:CwKa}, it follows that $T$ is also compact from
$[X_{0},X_{1}]_{\theta}=L^{p}(w^{p})$ to $[Y_{0},Y_{1}]_{\theta}=L^{q}(w^{q})$
as we wanted to show. 
\end{proof}
We devote the remainder of the section to settle Lemmas \ref{lem:Interp}
and \ref{lem:Interppq}.

\subsubsection{Proofs of the key lemmatta }

The first ingredient to settle Propositions \ref{lem:Interp} and
\ref{lem:Interppq} is the following Theorem.
\begin{thm}
If $q_{0},q_{1}\in[1,\infty)$ and $w_{0},w_{1}$ are two weights,
then for all $\theta\in(0,1)$ we have 
\[
[L^{q_{0}}(w_{0}),L^{q_{1}}(w_{1})]_{\theta}=L^{q}(w),
\]
where 
\begin{equation}
\frac{1}{q}=\frac{1-\theta}{q_{0}}+\frac{\theta}{q_{1}}\qquad\text{and}\qquad w^{\frac{1}{q}}=w_{0}^{\frac{1-\theta}{q_{0}}}w_{1}^{\frac{\theta}{q_{1}}}.\label{eq:convexity}
\end{equation}
\end{thm}

The result above can be found in of \cite[Theorem 5.5.3]{BLo}, however
it can be traced back to Stein and Weiss \cite{SW}. Having that result
in mind Lemmatta \ref{lem:Interp} and \ref{lem:Interppq} follow
from Lemmatta \ref{lem:Key} and \ref{lem:Keypq} respectively.
\begin{lem}
\label{lem:Key}Let $\lambda\in[1,\infty)$, $q_{0},q\in(\lambda,\infty)$,
$w_{1}\in A_{q_{1}/\lambda}^{-}$, $w\in A_{q/\lambda}^{+}.$ Then
there exists $q_{0}\in(\lambda,\infty)$, $w_{0}\in A_{q_{0}/\lambda}^{+}$,
and $\theta\in(0,1)$ such that (\ref{eq:convexity}) holds.
\end{lem}

\begin{lem}
\label{lem:Keypq}Let $1<p_{1}\leq q_{1}<\infty$, $1<p\leq q<\infty$,
$w_{1}\in A_{p_{1},q_{1}}^{-}$, $w\in A_{p,q}^{+}$. Then there exist
$1<p_{0}\leq q_{0}<\infty$, $w_{0}\in A_{p_{0},q_{0}}^{+}$, and
$\theta\in(0,1)$ such that
\begin{align*}
[L^{p_{0}}(w_{0}^{p_{0}}),L^{p_{1}}(w_{1}^{p_{1}})]_{\theta} & =L^{p}(w^{p})\qquad\text{and}\\{}
[L^{q_{0}}(w_{0}^{q_{0}}),L^{q_{1}}(w_{1}^{q_{1}})]_{\theta} & =L^{q}(w^{q})
\end{align*}
where 
\[
\frac{1}{p}=\frac{1-\theta}{p_{0}}+\frac{\theta}{p_{1}},\qquad\frac{1}{q}=\frac{1-\theta}{q_{0}}+\frac{\theta}{q_{1}},\qquad w=w_{0}^{1-\theta}w_{1}^{\theta}.
\]
\end{lem}

Before presenting the proofs of these Lemmatta we would like to provide
a remark regarding the hypothesis in the statement of Lemma \ref{lem:Key}.
A similar remark would is feasible as well for the $A_{p,q}^{+}$
case.
\begin{rem}
\label{rem:HyAp+-}Lemma \ref{lem:Key} does not hold if we replace
$w_{1}\in A_{q_{1}/\lambda}^{-}$ by $w_{1}\in A_{q_{1}/\lambda}^{+}$.
Indeed, assume that $w_{1}(x)=e^{x^{3}}\in A_{q}^{+}$ and that $w(x)=e^{x}\in A_{q}^{+}$.
Observe that if the conclusion of the preceding lemma held, that would
imply that $w_{0}=w_{1}^{-\frac{\theta}{q_{1}}\frac{q_{0}}{1-\theta}}w^{\frac{1}{q}\frac{q_{0}}{1-\theta}}$
for some $\theta\in(0,1)$, namely
\[
w_{0}(x)=e^{-x^{3}\frac{\theta}{q_{1}}\frac{q_{0}}{1-\theta}}e^{\frac{1}{q}\frac{q_{0}}{1-\theta}x}=e^{\frac{q_{0}}{1-\theta}\left(-\frac{\theta}{q_{1}}x^{3}+\frac{1}{q}x\right)}
\]
for some $\theta\in(0,1)$. Observe that for any $\theta\in(0,1)$
we have that 
\begin{equation}
\int_{0}^{\infty}w_{0}(x)dx<\infty.\label{eq:Condw0}
\end{equation}
However such a property prevents $w_{0}$ from being an $A_{q_{0}}^{+}$
weight for every possible choice of $q_{0},$ since if $w\in A_{q_{0}}^{+}$
then 
\[
\int_{a}^{\infty}w_{0}(x)dx=\infty
\]
for every $a\in\mathbb{R}$. Indeed, if $w\in A_{q_{0}}^{+}$ then
\[
\int_{0}^{x}w_{0}dy\leq c\int_{x}^{2x}w_{0}dy
\]
if additionally (\ref{eq:Condw0}) holds, the right hand side is uniformly
controlled by (\ref{eq:Condw0}). Hence by dominated convergence theorem
\[
\int_{0}^{\infty}w_{0}(x)dx=\lim_{x\rightarrow\infty}\int_{0}^{x}w_{0}dy=0
\]
which is a contradiction since $w_{0}$ is positive a.e.
\end{rem}

The remainder of this subsection is devoted to settle Lemmatta (\ref{lem:Key})
and (\ref{lem:Keypq}). For that purpose we will use the following
notation. Given any interval $I=(a,b)$, we denote 
\[
I_{1}^{n}=\left(a,a+\frac{l_{I}}{n}\right),\qquad I_{2}^{n}=\left(b-\frac{l_{I}}{n},b\right).
\]

\subsubsection{Proof of Lemma \ref{lem:Key} }

We adapt the argument in \cite{HL} to the one-sided setting. Note
that given $\theta\in(0,1)$ we would need to show that $w_{0}\in A_{q_{0}}^{+}$
where
\[
q_{0}(\theta)=\frac{1-\theta}{\frac{1}{q}-\frac{\theta}{q_{1}}}\qquad w^{\frac{q_{0}(\theta)}{q(1-\theta)}}w_{1}^{-\frac{\theta q_{0}(\theta)}{q_{1}(1-\theta)}}=w_{0}(\theta)=w_{0}.
\]
Note that by the definition of $q$ and $q_{1}$ we have that $q_{0}\in(\lambda,\infty)$.
Observe that calling $p_{i}=\frac{q_{i}}{\lambda}$ and $p=\frac{q}{\lambda}$,
the same relations hold, namely
\[
p_{0}=\frac{q_{0}}{\lambda}=\frac{1-\theta}{\frac{\lambda}{q}-\frac{\lambda\theta}{q_{1}}}=\frac{1-\theta}{\frac{1}{p}-\frac{\theta}{p_{1}}}
\]
and hence it suffices to show that
\[
\left[w_{0}\right]_{A_{p_{0}}^{+}}=\sup_{a<x_{1}<b}\left(\frac{1}{|(a,b)|}\int_{a}^{x_{1}}w_{0}\right)\left(\frac{1}{|(a,b)|}\int_{x_{1}}^{b}w_{0}^{-\frac{1}{p_{0}-1}}\right)^{p_{0}-1}<\infty.
\]
The remainder of the argument reduces to settle the following claim
\begin{claim}
\label{claim:2}Given an interval $I$ we have that
\[
\left(\frac{1}{|I_{1}^{4}|}\int_{I_{1}^{4}}w_{0}\right)\left(\frac{1}{|I_{2}^{4}|}\int_{I_{2}^{4}}w_{0}^{-\frac{1}{p_{0}-1}}\right)^{p_{0}-1}\leq[w]_{A_{p}}^{\frac{p_{1}}{p_{1}-\theta p}}[w_{1}]_{A_{p_{1}}}^{\frac{\theta p}{p_{1}-\theta p}}<\infty
\]
with $\theta\in(0,1)$ small enough and independent of $I$. 
\end{claim}

Note that the desired result readily follows combining this claim
with a direct application of Lemma \ref{lem:Gap} and hence, it suffices
to settle Claim \ref{claim:2} to end the proof. From this point it
suffices to proceed in \cite[Lemma 4.1]{HL} to our situation. Since
there are some differences in the way we need to apply reverse Hölder
inequality, we provide the full argument for reader's convenience.
For some $\varepsilon,\delta>0$ to be chosen we have that
\begin{align*}
 & \left(\frac{1}{|I_{1}^{4}|}\int_{I_{1}^{4}}w_{0}\right)\left(\frac{1}{|I_{2}^{4}|}\int_{I_{2}^{4}}w_{0}^{-\frac{1}{p_{0}-1}}\right)^{p_{0}-1}\\
 & =\left(\frac{1}{|I_{1}^{4}|}\int_{I_{1}^{4}}w^{\frac{p_{0}}{p(1-\theta)}}w_{1}^{-\frac{p_{0}\theta}{p_{1}(1-\theta)}}\right)\left(\frac{1}{|I_{2}^{4}|}\int_{I_{2}^{4}}w^{-\frac{p'_{0}}{p(1-\theta)}}w_{1}^{\frac{p'_{0}\theta}{p_{1}(1-\theta)}}\right)^{p_{0}-1}\\
 & =\left(\frac{1}{|I_{1}^{4}|}\int_{I_{1}^{4}}w^{\frac{p_{0}}{p(1-\theta)}}\left(w_{1}^{-\frac{1}{p_{1}-1}}\right)^{\frac{p_{0}\theta}{p'_{1}(1-\theta)}}\right)\left(\frac{1}{|I_{2}^{4}|}\int_{I_{2}^{4}}\left(w^{-\frac{1}{p-1}}\right)^{\frac{p'_{0}}{p'(1-\theta)}}w_{1}^{\frac{p'_{0}\theta}{p_{1}(1-\theta)}}\right)^{p_{0}-1}\\
 & \leq\left(\frac{1}{|I_{1}^{4}|}\int_{I_{1}^{4}}w^{\frac{p_{0}(1+\varepsilon)}{p(1-\theta)}}\right)^{\frac{1}{1+\varepsilon}}\left(\frac{1}{|I_{1}^{4}|}\int_{I_{1}^{4}}\left(w_{1}^{-\frac{1}{p_{1}-1}}\right)^{\frac{p_{0}\theta(1+\varepsilon)}{p_{1}^{'}\varepsilon(1-\theta)}}\right)^{\frac{\varepsilon}{1+\varepsilon}}\\
 & \times\left(\frac{1}{|I_{2}^{4}|}\int_{I_{2}^{4}}\left(w^{-\frac{1}{p-1}}\right)^{\frac{p_{0}^{'}(1+\delta)}{p'(1-\theta)}}\right)^{\frac{p_{0}-1}{1+\delta}}\left(\frac{1}{|I_{2}^{4}|}\int_{I_{2}^{4}}w_{1}^{\frac{p_{0}^{'}\theta(1+\delta)}{p_{1}\delta p'(1-\theta)}}\right)^{\delta\frac{p_{0}-1}{1+\delta}}\\
 & =\left(\frac{1}{|I_{1}^{4}|}\int_{I_{1}^{4}}w^{r(\theta)}\right)^{\frac{1}{1+\varepsilon}}\left(\frac{1}{|I_{1}^{4}|}\int_{I_{1}^{4}}\left(w_{1}^{-\frac{1}{p_{1}-1}}\right)^{s(\theta)}\right)^{\frac{\varepsilon}{1+\varepsilon}}\\
 & \times\left(\frac{1}{|I_{2}^{4}|}\int_{I_{2}^{4}}\left(w^{-\frac{1}{p-1}}\right)^{t(\theta)}\right)^{\frac{p_{0}-1}{1+\delta}}\left(\frac{1}{|I_{2}^{4}|}\int_{I_{2}^{4}}w_{1}^{u(\theta)}\right)^{\delta\frac{p_{0}-1}{1+\delta}}=(*).
\end{align*}
where
\begin{align*}
r(\theta) & =\frac{p_{0}(\theta)(1+\varepsilon)}{p(1-\theta)}\qquad s(\theta)=\frac{p_{0}(\theta)\theta(1+\varepsilon)}{p_{1}^{'}\varepsilon(1-\theta)}\\
t(\theta) & =\frac{p_{0}^{'}(\theta)(1+\delta)}{p'(1-\theta)}\qquad u(\theta)=\frac{p'_{0}(\theta)\theta(1+\delta)}{p_{1}\delta(1-\theta)}.
\end{align*}
Choosing $\varepsilon=\frac{\theta p}{p_{1}'}$ and $\delta=\frac{\theta p'}{p_{1}}$
we have that 
\begin{align*}
r(\theta) & =s(\theta)=\frac{p_{0}(\theta)(p'_{1}+\theta p')}{pp'_{1}(1-\theta)},\\
t(\theta) & =u(\theta)=\frac{p_{0}(\theta)(p{}_{1}+\theta p')}{p'p_{1}(1-\theta)}.
\end{align*}
Note that by Lemma \ref{lem:LemRHI} there exists $\gamma_{1}>1$
such that 
\[
\left(\frac{1}{|I_{1}^{4}|}\int_{I_{1}^{4}}\rho^{\gamma_{1}}\right)^{\frac{1}{\gamma_{1}}}\leq4\frac{1}{|I_{1}^{2}|}\int_{I_{1}^{2}}\rho
\]
where $\rho=w$ and $\rho=w_{1}^{-\frac{1}{p_{1}-1}}$ are $A_{\infty}^{+}$
weights and also there exists $\gamma_{2}>1$ such that 
\[
\left(\frac{1}{|I_{2}^{4}|}\int_{I_{2}^{4}}\rho^{\gamma_{1}}\right)^{\frac{1}{\gamma_{1}}}\leq4\frac{1}{|I_{2}^{2}|}\int_{I_{2}^{2}}\rho
\]
where $\rho=w_{1}$ and $\rho=w^{-\frac{1}{p-1}}$ are $A_{\infty}^{-}$
weights. Note that $p_{0}(0)=p$ and that $r(0)=t(0)=1$. Hence, by
continuity it is clear that we can find $\theta\in(0,1)$ small enough
such that 
\[
1<\max\left\{ r(\theta),t(\theta)\right\} \leq\min\{\gamma_{1},\gamma_{2}\}.
\]
For that choice of $\theta$ we can continue the argument as follows.
Applying Lemma \ref{lem:LemRHI} we have that

\begin{align*}
(*) & \lesssim\left(\frac{1}{|I_{1}^{2}|}\int_{I_{1}^{2}}w\right)^{\frac{p_{0}}{p(1-\theta)}}\left(\frac{1}{|I_{1}^{2}|}\int_{I_{1}^{2}}w_{1}^{-\frac{1}{p_{1}-1}}\right)^{\frac{\theta p_{0}}{p_{1}'(1-\theta)}}\\
 & \times\left(\frac{1}{|I_{2}^{2}|}\int_{I_{2}^{2}}w^{-\frac{1}{p-1}}\right)^{\frac{p_{0}^{'}(p_{0}-1)}{p'(1-\theta)}}\left(\frac{1}{|I_{2}^{2}|}\int_{I_{2}^{2}}w_{1}\right)^{\frac{\theta p_{0}'(p_{0}-1)}{p_{1}(1-\theta)}}=(**)
\end{align*}
and finally, rearranging terms,
\begin{align*}
(**) & \lesssim\left(\left(\frac{1}{|I_{1}^{2}|}\int_{I_{1}^{2}}w\right)\left(\frac{1}{|I_{2}^{2}|}\int_{I_{2}^{2}}w^{-\frac{1}{p-1}}\right)^{p-1}\right)^{\frac{p_{0}}{p(1-\theta)}}\\
 & \times\left(\left(\frac{1}{|I_{2}^{2}|}\int_{I_{2}^{2}}w_{1}\right)\left(\frac{1}{|I_{1}^{2}|}\int_{I_{1}^{2}}w_{1}^{-\frac{1}{p_{1}-1}}\right)^{p_{1}-1}\right)^{\frac{\theta p_{0}}{p_{1}(1-\theta)}}\\
 & \leq[w]_{A_{p}^{+}}^{\frac{p_{1}}{p_{1}-\theta p}}[w_{1}]_{A_{p_{1}}^{-}}^{\frac{\theta p}{p_{1}-\theta p}}
\end{align*}
as we wanted to show.

\subsubsection{Proof of Lemma \ref{lem:Keypq}}

Observe that by the choice of $\theta\in(0,1)$ we have that $p_{0}$,
$q_{0}$ and $w_{0}$ are determined as follows
\begin{align*}
p_{0} & =p_{0}(\theta)=\frac{1-\theta}{\frac{1}{p}-\frac{\theta}{p_{1}}}\\
q_{0} & =q_{0}(\theta)=\frac{1-\theta}{\frac{1}{q}-\frac{\theta}{q_{1}}}\\
w_{0} & =w_{0}(\theta)=w^{\frac{1}{1-\theta}}w_{1}^{-\frac{\theta}{1-\theta}}
\end{align*}
and consequently it remains to show that it is possible to choose
$\theta\in(0,1)$ such that $1<p_{0}\leq q_{0}<\infty$ and $w_{0}\in A_{p_{0},q_{0}}^{+}$.
Due to the fact that $1<p_{0}(0)=p\leq q=q_{0}(0)<\infty$, it is
clear that choosing $\theta$ small enough the first condition holds,
and hence we are left with fulfilling the second one. We will follow
the same strategy used in the proof of Lemma \ref{lem:Key}, namely,
given an interval $I=(a,b)$ we are going to show that
\[
\left(\frac{1}{|I_{1}^{4}|}\int_{I_{1}^{4}}w_{0}^{q_{0}}\right)^{\frac{1}{q_{0}}}\left(\frac{1}{|I_{2}^{4}|}\int_{I_{2}^{4}}w_{0}^{-p_{0}'}\right)^{\frac{1}{p_{0}'}}\leq[w]_{A_{p,q}^{+}}^{\frac{1}{1-\theta}}[w]_{A_{p_{1},q_{1}}^{-}}^{\frac{\theta}{1-\theta}}
\]
since relying upon this estimate, a direct application of Lemma \ref{lem:Gap}
allows to derive the desired conclusion.

We argue as follows using Hölder inequality with exponents $1+\varepsilon$
and $1+\delta$ with $\delta,\varepsilon>0$ to be chosen.
\begin{align}
 & \left(\frac{1}{|I_{1}^{4}|}\int_{I_{1}^{4}}w_{0}^{q_{0}}\right)^{\frac{1}{q_{0}}}\left(\frac{1}{|I_{2}^{4}|}\int w_{0}^{-p_{0}'}\right)^{\frac{1}{p_{0}'}}\nonumber \\
= & \left(\frac{1}{|I_{1}^{4}|}\int_{I_{1}^{4}}w^{\frac{q_{0}}{1-\theta}}w_{1}^{-\frac{\theta q_{0}}{1-\theta}}\right)^{\frac{1}{q_{0}}}\left(\frac{1}{|I_{2}^{4}|}\int_{I_{2}^{4}}w^{-p_{0}'\frac{1}{1-\theta}}w_{1}^{p_{0}'\frac{\theta}{1-\theta}}\right)^{\frac{1}{p_{0}'}}\nonumber \\
\leq & \left(\frac{1}{|I_{1}^{4}|}\int_{I_{1}^{4}}w^{\frac{q_{0}(1+\varepsilon)}{1-\theta}}\right)^{\frac{1}{q_{0}(1+\varepsilon)}}\left(\frac{1}{|I_{1}^{4}|}\int_{I_{1}^{4}}w_{1}^{-\frac{\theta q_{0}}{1-\theta}\frac{1+\varepsilon}{\varepsilon}}\right)^{\frac{1}{q_{0}}\frac{\varepsilon}{1+\varepsilon}}\nonumber \\
\times & \left(\frac{1}{|I_{2}^{4}|}\int_{I_{2}^{4}}w^{-p_{0}'\frac{1}{1-\theta}(1+\delta)}\right)^{\frac{1}{p_{0}'(1+\delta)}}\left(\frac{1}{|I_{2}^{4}|}\int_{I_{2}^{4}}w_{1}^{p_{0}'\frac{\theta}{1-\theta}\frac{1+\delta}{\delta}}\right)^{\frac{1}{p_{0}'}\frac{\delta}{1+\delta}}\nonumber \\
= & \left(\frac{1}{|I_{1}^{4}|}\int_{I_{1}^{4}}w^{qr(\theta)}\right)^{\frac{1}{q_{0}(1+\varepsilon)}}\left(\frac{1}{|I_{1}^{4}|}\int_{I_{1}^{4}}w_{1}^{-p_{1}^{'}s(\theta)}\right)^{\frac{1}{q_{0}}\frac{\varepsilon}{1+\varepsilon}}\label{eq:rh1}\\
\times & \left(\frac{1}{|I_{2}^{4}|}\int_{I_{2}^{4}}w^{-p't(\theta)}\right)^{\frac{1}{p_{0}'(1+\delta)}}\left(\frac{1}{|I_{2}^{4}|}\int_{I_{2}^{4}}w_{1}^{q_{1}u(\theta)}\right)^{\frac{1}{p_{0}'}\frac{\delta}{1+\delta}}\label{eq:rh2}
\end{align}
where
\begin{align*}
r(\theta) & =\frac{q_{0}(\theta)(1+\varepsilon)}{q(1-\theta)}\\
s(\theta) & =\frac{1}{p_{1}'}\frac{\theta q_{0}(\theta)}{(1-\theta)}\frac{1+\varepsilon}{\varepsilon}\\
t(\theta) & =\frac{p_{0}'(\theta)}{p'}\frac{1}{1-\theta}(1+\delta)\\
u(\theta) & =\frac{p_{0}'(\theta)}{q_{1}}\frac{\theta}{1-\theta}\frac{1+\delta}{\delta}.
\end{align*}
Note that choosing $\varepsilon=\frac{q\theta}{p'_{1}}$ and $\delta=\frac{\theta p'}{q_{1}}$
we have that 
\begin{align*}
r(\theta)=s(\theta) & =\frac{q_{0}(\theta)(p_{1}'+\theta q)}{qp_{1}'(1-\theta)}\\
t(\theta)=u(\theta) & =\frac{p_{0}'(\theta)(q_{1}+\theta p')}{q_{1}p'(1-\theta)}.
\end{align*}
Now we observe that by Proposition \ref{prop:Apq} $w^{q}\in A_{1+\frac{q}{p'}}^{+}$,
$w^{-p'}\in A_{1+\frac{p'}{q}}^{-}$, $w_{1}^{-p_{1}'}\in A_{1+\frac{p'_{1}}{q_{1}}}^{+}$
and $w_{1}^{q_{1}}\in A_{1+\frac{q_{1}}{p'_{1}}}^{-}$, we have that
$w^{q},w^{-p'_{1}}\in A_{\infty}^{+}$ and $w^{-p'},w_{1}^{q_{1}}\in A_{\infty}^{-}$.
Since for any $\eta>0$, it is always possible to choose $\theta$
close enough to $0$ such that
\[
1<\max\left\{ r(\theta),t(\theta)\right\} \leq1+\eta
\]
making a suitable choice of $\theta$ allows to apply reverse Hölder
inequalities in Lemma \ref{lem:LemRHI} to each term in (\ref{eq:rh1})
and in (\ref{eq:rh2}) to obtain
\begin{align*}
 & \left(\frac{1}{|I_{1}^{4}|}\int_{I_{1}^{4}}w^{qr(\theta)}\right)^{\frac{1}{q_{0}(1+\varepsilon)}}\left(\frac{1}{|I_{1}^{4}|}\int_{I_{1}^{4}}w_{1}^{-p_{1}^{'}s(\theta)}\right)^{\frac{1}{q_{0}}\frac{\varepsilon}{1+\varepsilon}}\\
\times & \left(\frac{1}{|I_{2}^{4}|}\int_{I_{2}^{4}}w^{-p't(\theta)}\right)^{\frac{1}{p_{0}'(1+\delta)}}\left(\frac{1}{|I_{2}^{4}|}\int_{I_{2}^{4}}w_{1}^{q_{1}u(\theta)}\right)^{\frac{1}{p_{0}'}\frac{\delta}{1+\delta}}\\
\lesssim & \left(\frac{1}{|I_{1}^{2}|}\int_{I_{1}^{2}}w^{q}\right)^{\frac{1}{q_{0}(1+\varepsilon)}r(\theta)}\left(\frac{1}{|I_{1}^{2}|}\int_{I_{1}^{2}}w_{1}^{-p_{1}^{'}s(\theta)}\right)^{\frac{1}{q_{0}}\frac{\varepsilon}{1+\varepsilon}r(\theta)}\\
\times & \left(\frac{1}{|I_{2}^{2}|}\int_{I_{2}^{2}}w^{-p'}\right)^{\frac{1}{p_{0}'(1+\delta)}t(\theta)}\left(\frac{1}{|I_{2}^{2}|}\int_{I_{2}^{2}}w_{1}^{q_{1}}\right)^{\frac{1}{p_{0}'}\frac{\delta}{1+\delta}t(\theta)}\\
= & \left(\frac{1}{|I_{1}^{2}|}\int_{I_{1}^{2}}w^{q}\right)^{\frac{1}{q(1-\theta)}}\left(\frac{1}{|I_{2}^{2}|}\int_{I_{2}^{2}}w^{-p'}\right)^{\frac{1}{p'(1-\theta)}}\\
\times & \left(\frac{1}{|I_{2}^{2}|}\int_{I_{2}^{2}}w_{1}^{q_{1}}\right)^{\frac{1}{q_{1}}\frac{\theta}{1-\theta}}\left(\frac{1}{|I_{1}^{2}|}\int_{I_{1}^{2}}w_{1}^{-p_{1}^{'}s(\theta)}\right)^{\frac{1}{p_{1}'}\frac{\theta}{1-\theta}}\\
\leq & [w]_{A_{p,q}^{+}}^{\frac{1}{1-\theta}}[w]_{A_{p_{1},q_{1}}^{-}}^{\frac{\theta}{1-\theta}}
\end{align*}
as we wanted to show.

\section{Applications\label{sec:Applications}}

\subsection{A result on the compactness of one sided Calderón-Zygmund operators}

In \cite[Theorem 2.21]{V} Villarroya showed that a Calderón-Zygmund
operator $T$ extends to a compact operator on $L^{p}(\mathbb{R})$
for $1<p<\infty$ if and only if $T$ is associated with a compact
Calderón-Zygmund kernel \cite[Definition 2.3]{V} and satisfies the
weak compactness condition \cite[Definition 2.12]{V}, and $T(1),T^{*}(1)\in CMO(\mathbb{R})$. 

The aforementioned result combined with Corollary \ref{thm:CorUnWeighted}
leads to the following result.
\begin{cor}
Let $T$ be a Calderón-Zygmund operator associated to a compact Calderón-Zygmund
kernel supported in $\{(x,y)\in\mathbb{R}^{2}\,:\,x<y\}$ such that
$T$ satisfies the weak compactness condition and $T(1),T^{*}(1)\in CMO(\mathbb{R})$.
Then for every $1<p<\infty$, $T$ is compact on $L^{p}(w)$ with
$w\in A_{p}^{+}$.
\end{cor}

\begin{proof}
Note that in \cite{AFMR,CKr} it was shown that a Calderón-Zygmund
with kernel supported on $\{(x,y)\in\mathbb{R}^{2}\,:\,x<y\}$ is
bounded on $L^{p}(w)$ for $w\in A_{p}^{+}$ and every $p\in(1,\infty)$.
This fact combined with Villarroya's \cite[Theorem 2.21]{V}, which
ensures the compactness in the unweighted setting, leads via Corollary
\ref{thm:CorUnWeighted} to the desired conclusion.
\end{proof}
Very recently Mitkovski and Stockdale showed that a Calderón-Zygmund
operator $T$ extends to a compact operator on $L^{2}(\mathbb{R})$
if and only if $T$ is weakly compact \cite[Section 2]{MS}, and $T(1),T^{*}(1)\in CMO(\mathbb{R})$.
Arguing as above we have the following Corollary.
\begin{cor}
Let $T$ be a Calderón-Zygmund operator associated to a compact Calderón-Zygmund
kernel supported in $\{(x,y)\in\mathbb{R}^{2}\,:\,x<y\}$ such that
$T$ is weakly compact and $T(1),T^{*}(1)\in CMO(\mathbb{R})$. Then
for every $1<p<\infty$, $T$ is compact on $L^{p}(w)$ with $w\in A_{p}^{+}$.
\end{cor}

\subsection{Results on commutators}

\subsubsection{Commutators of $L^{r}$- Hörmander operators}

In this section we will work in the $n$-dimensional case in the classical
setting, since the results we present are unknown even that setting.

Before presenting the main result of this section we need the following
definition that we essentially borrow from \cite{BCADH}. We remit
the reader there for further references related to this definition
in the literature.
\begin{defn}
We say that a locally integrable function $K:\{(x,y)\in\mathbb{R}^{n}\,:\,x\not=y\}\rightarrow\mathbb{R}$
is an $L^{r}$-Hörmander kernel if there exists $\gamma>n/r'$ such
that 
\[
\left(\int_{A_{m}(B)}|K(x,y)-K(x',y)|^{r}dy\right)^{\frac{1}{r}}+\left(\int_{A_{m}(B)}|K(y,x)-K(y,x')|^{r}dy\right)^{\frac{1}{r}}\leq C\frac{|x-x'|^{\gamma-\frac{n}{r'}}}{|B|^{\frac{\gamma}{n}}}2^{-m\gamma}
\]
for every ball $B$ and all $x,x'\in\frac{1}{2}B$ where $A_{m}(B)=2^{m}B\setminus2^{m-1}B$
with $m\geq1$, and it also satisfies the size condition
\[
|K(x,y)|\leq C\frac{1}{|x-y|^{n}}.
\]
We say that a linear operator $T$ is an $L^{r}$-Hörmander singular
integral operator operator if $T$ is bounded on $L^{2}$, there exists
an $L^{r}$-Hörmander kernel $K$ such that for every $f\in L_{c}^{\infty}(\mathbb{R}^{n})$,
\[
Tf(x)=\int_{\mathbb{R}^{n}}K(x,y)f(y)dy\qquad x\not\in\text{supp}(f)
\]
and also $\lim_{\varepsilon\rightarrow0^{+}}\int_{\varepsilon<|x-y|<1}K(x,y)dy$
exists.
\end{defn}

Note that under the conditions above, in fact under weaker ones, Grafakos
showed \cite{G} that if $T$ is an $L^{r}$-Hörmander operator with
associated kernel $K$, then the maximal operator
\[
T^{*}f(x)=\sup_{\varepsilon>0}\left|\int_{\varepsilon<|x-y|}K(x,y)f(y)dy\right|
\]
is bounded on $L^{p}$. Consequently if $T$ is an $L^{r}$-Hörmander
operator, we have that 
\[
Tf(x)=\lim_{\varepsilon\rightarrow0^{+}}\int_{\varepsilon<|x-y|}K(x,y)f(y)dy
\]
almost everywhere for $f\in L^{p}$.

Our result of compactness that as we announced in the abstract seems
to be new even in the classical setting is the following.
\begin{thm}
\label{thm:compLrHor}Let $1<r'<p<\infty$. Assume that $K$ is an
$L^{r}$-Hörmander kernel and that $T$ stands for the operator associated
to $K$. If $b\in CMO$, then for every $w\in A_{p/r'}$, $[b,T]$
is compact from $L^{p}(w)$ to $L^{p}(w)$. If $n=1$ and additionally,
$K$ is supported on $\{(x,y)\in\mathbb{R}^{2}\,:\,x<y\}$, then we
have that for every $w\in A_{p/r'}^{+}$, $[b,T]$ is compact on $L^{p}(w)$.
\end{thm}

Theorem \ref{thm:compLrHor} follows from \cite{HL} in the standard
setting, and from Theorem \ref{thm:CorUnWeighted} in the one sided
setting, combined with the following result that we settle in this
paper.
\begin{thm}
\label{thm:compLrHorUnweighted}Let $1<r'<p<\infty$. Let $T$ be
an $L^{r}$-Hörmander operator. If $b\in CMO$, then $[b,T]$ is compact
on $L^{p}$.
\end{thm}

We will present both the proof of Theorem \ref{thm:compLrHorUnweighted}
and of Theorem (\ref{thm:compLrHor}) in the next subsection.

\subsubsection{Commutators of one sided fractional integrals}

We recall that given $\alpha\in(0,1)$ the one sided fractional integral
$I_{\alpha}^{+}$ is defined as
\[
I_{\alpha}^{+}f(x)=\int_{x}^{\infty}\frac{f(y)}{(y-x)^{1-\alpha}}dy.
\]
Our result deals with the compactness of the commutator of this operator.
The statement is the following.
\begin{thm}
\label{thm:compFrac}Let $1<p<q<\infty$ and $\alpha\in(0,1)$ such
that $\frac{1}{p}-\frac{1}{q}=\alpha$. If $b\in CMO$ and $w\in A_{p}^{+}$
then $[b,I_{\alpha}^{+}]$ is compact from $L^{p}(w)$ to $L^{q}(w)$. 
\end{thm}

The theorem above readily follows from Theorem \ref{thm:Extrapolationpq}
combined with the following result that is a consequence of \cite[Theorem 1.5]{GWY}
as we shall show.
\begin{thm}
\label{thm:CompactFracUnw}Let $1<p<q<\infty$ and $\alpha\in(0,1)$
such that $\frac{1}{p}-\frac{1}{q}=\alpha$. If $b\in CMO$ and $w\in A_{p}^{+}$
then $[b,I_{\alpha}^{+}]$ is compact from $L^{p}$ to $L^{q}$. 
\end{thm}

We give the proofs of these results in the next subsection.

\subsection{Proofs of the results on commutators}

\subsubsection{Proof of Theorem \ref{thm:compLrHorUnweighted}}

A result that will be useful for our purposes appears in \cite{Y}.
\begin{lem}
\label{lem:Yosida}Let $1\leq p<\infty$. A subset $H\subset L^{p}(\mathbb{R}^{n})$
is precompact if and only if the following conditions hold.
\begin{enumerate}
\item $H$ is bounded in $L^{p}(\mathbb{R})$
\item $\lim_{h\rightarrow0}\int_{\mathbb{R}}|f(x+h)-f(x)|^{p}dx=0$ uniformly
in $H$, 
\item $\lim_{M\rightarrow0}\int_{|x|>M}|f|^{p}=0$ uniformly in $H$,
\end{enumerate}
\end{lem}

Having the result above at our disposal, before getting into the proof
of Theorem \ref{thm:compLrHorUnweighted} we need some preparations. 

Let $\varphi:\mathbb{R}^{n}\rightarrow[0,1]$ be a differentiable
radial function function such that 
\begin{align*}
\supp(\varphi^{\delta}) & \subset\{|x|>\delta\}\\
\varphi^{\delta}(x) & =1\qquad|x|\geq2\delta\\
\|\varphi^{\delta}\|_{L^{\infty}} & =1\\
\|\nabla\varphi^{\delta}\|_{L^{\infty}} & \simeq\frac{1}{\delta}
\end{align*}
Assume that $K$ satisfies and $L^{r}$-Hörmander condition. For each
$\delta\in(0,1)$ we shall call $K^{\delta}(x,y)=\varphi^{\delta}(x-y)K(x,y)$.
In the following lines we study the $L^{r}$-Hörmander conditions
for $K^{\delta}$. 
\begin{lem}
\label{lem:Trunc}Let $r>1$ and $\delta\in(0,1)$. If $K$ satisfies
the $L^{r}$-Hörmander condition for some $\gamma>n/r'$, then 
\[
|K^{\delta}(x,y)|\leq C\frac{1}{|x-y|^{n}}\qquad x\not=y,
\]
and also
\begin{align*}
 & \left(\int_{A_{m}(B)}|K^{\delta}(x,y)-K^{\delta}(x',y)|^{r}dy\right)^{\frac{1}{r}}+\left(\int_{A_{m}(B)}|K^{\delta}(y,x)-K^{\delta}(y,x')|^{r}dy\right)^{\frac{1}{r}}\\
 & \leq C\max\left\{ \frac{|x-x'|^{\gamma-\frac{n}{r'}}}{|B|^{\frac{\gamma}{n}}}2^{-m\gamma},\frac{|x-x'|^{(1+\frac{n}{r'})-\frac{n}{r'}}}{|B|^{(1+\frac{n}{r'})\frac{1}{n}}}2^{-(1+\frac{n}{r'})m}\right\} 
\end{align*}
for every ball $B$ and all $x,x'\in\frac{1}{2}B$ where $A_{m}(B)=2^{m}B\setminus2^{m-1}B$
if $m\geq1$.
\end{lem}

\begin{proof}
It is trivial that $K^{\delta}$ satisfies the size condition. For
the smoothness condition it suffices to work with just one of the
terms, since the condition is symmetric.
\[
\left(\int_{A_{m}(B)}|K^{\delta}(x,y)-K^{\delta}(x',y)|^{r}dy\right)^{\frac{1}{r}}\leq\sum_{i=1}^{3}\left(\int_{C_{j}^{m}(B)}|K^{\delta}(x,y)-K^{\delta}(x',y)|^{r}dy\right)^{\frac{1}{r}}
\]
where 
\begin{align*}
C_{1}^{m} & =A_{m}(B)\cap\{|x-y|>2\delta\}\cap\{|x'-y|>2\delta\}\\
C_{2}^{m} & =A_{m}(B)\cap\{|x-y|<2\delta\}\cap\{|x'-y|>2\delta\}\\
C_{3}^{m} & =A_{m}(B)\cap\{|x-y|>2\delta\}\cap\{|x'-y|<2\delta\}
\end{align*}
We begin observing that 
\[
\left(\int_{C_{1}^{m}(B)}|K^{\delta}(x,y)-K^{\delta}(x',y)|^{r}dy\right)^{\frac{1}{r}}\leq\left(\int_{A_{m}(B)}|K(x,y)-K(x',y)|^{r}dx\right)^{\frac{1}{r}}
\]
and hence the desired conclusion holds for this term. Since $C_{2}^{m}$
and $C_{3}^{m}$ are symmetric, hence it suffices to deal with $C_{2}^{m}$.
We argue as follows. Note that 
\begin{align*}
\left(\int_{C_{2}^{m}}|K^{\delta}(x,y)-K^{\delta}(x',y)|^{r}dy\right)^{\frac{1}{r}} & =\left(\int_{C_{2}^{m}}|K(x,y)\varphi^{\delta}(x-y)-K(x',y)\varphi^{\delta}(x'-y)|^{r}dy\right)^{\frac{1}{r}}\\
 & \leq\left(\int_{C_{2}^{m}}|K(x,y)\varphi^{\delta}(x-y)-K(x',y)\varphi^{\delta}(x-y)|^{r}dy\right)^{\frac{1}{r}}\\
 & +\left(\int_{C_{2}^{m}}|K(x',y)(\varphi^{\delta}(x'-y)-\varphi^{\delta}(x-y))|^{r}dy\right)^{\frac{1}{r}}\\
 & =I+II
\end{align*}
For $I$ since $\|\varphi\|_{L^{\infty}}\leq1$, 
\[
I\leq\left(\int_{A_{m}(B)}|K(x,y)-K(x',y)|^{r}dy\right)^{\frac{1}{r}}
\]
and we are done using $L^{r}$-Hörmander condition for $K$. 

For $II$ we have that 
\begin{align*}
II & \leq\left(2^{r}|x-x'|^{r}\|\nabla\varphi\|_{L^{\infty}}^{r}\int_{C_{2}^{m}}|K(x',y)|^{r}dy\right)^{\frac{1}{r}}\\
 & \leq C|x-x'|\left(\frac{1}{\delta^{r}}\int_{C_{2}^{m}}|K(x',y)|^{r}dy\right)^{\frac{1}{r}}\\
 & \leq C|x-x'|\left(\frac{1}{\delta^{r}}\int_{C_{2}^{m}}\frac{1}{|x'-y|^{nr}}dy\right)^{\frac{1}{r}}
\end{align*}
For $y\in C_{2}^{m}$ we have that 
\[
\frac{1}{\delta}<\frac{2}{|x-y|}
\]
and also for every $z\in\frac{1}{2}B$ and $y\in A_{m}(B)$, if we
denote $c_{B}$ the center of the ball $B$, 
\[
\frac{1}{|z-y|}\leq\frac{2}{|c_{B}-y|}.
\]
Hence, calling $r_{B}$ the radius of the ball $B$, 
\begin{align*}
II & \leq C|x-x'|\left(\int_{C_{2}^{m}}\frac{1}{|c_{B}-y|^{(n+1)r}}dy\right)^{\frac{1}{r}}\leq C|x-x'|\left(\int_{A_{m}(B)}\frac{1}{|c_{B}-y|^{(n+1)r}}dy\right)^{\frac{1}{r}}\\
 & =C|x-x'|\left(\int_{2^{m-1}r_{B}}^{2^{m}r_{B}}\frac{1}{\rho^{(n+1)r}}\rho^{n-1}d\rho\right)^{\frac{1}{r}}=C|x-x'|\left(\left[\frac{\rho^{n-(n+1)r}}{n-(n+1)r}\right]_{2^{m-1}r_{B}}^{2^{m}r_{B}}\right)^{\frac{1}{r}}\\
 & \leq C|x-x'|(2^{m}r_{B})^{\frac{n}{r}-(n+1)}=|x-x'|(2^{m}r_{B})^{-\frac{n}{r'}-1}\\
 & =\frac{|x-x'|}{2}r_{B}{}^{-\frac{n}{r'}}2^{-m}r_{B}^{-1}2^{-\frac{n}{r'}m}\simeq\frac{|x-x'|^{(1+\frac{n}{r'})-\frac{n}{r'}}}{|B|^{(1+\frac{n}{r'})\frac{1}{n}}}2^{-(1+\frac{n}{r'})m}.
\end{align*}
and hence we are done.
\end{proof}
Armed with the lemma above we can present our proof of Theorem \ref{thm:compLrHor}.

We begin with some reductions. Note that since for every $p>r'$ 
\[
\|[b,T]f\|_{L^{p}}\leq C\|b\|_{BMO}\|f\|_{L^{p}}
\]
(see for instance \cite{LMRdT}) we have that if $b\in CMO$ then
we can approximate $b$ by functions $b_{j}\in C_{c}^{\infty}$ in
the $BMO$ norm and consequently
\[
\|[b,T]f-[b_{j},T]f\|_{L^{p}}=\|[b-b_{j},T]f\|_{L^{p}}\lesssim\|b-b_{j}\|_{BMO}\|f\|_{L^{p}}.
\]
In particular $[b_{j},T]\rightarrow[b,T]$ in the $L^{p}$norm. Consequently
it suffices to prove the compactness for the commutator with smooth
symbol. Now we shall consider $T^{\delta}$ the operator associated
to the truncation defined as above with $\delta\in(0,1)$.
\begin{lem}
Let $b\in C_{c}^{1}(\mathbb{R}^{n})$. There exists a constant $C>0$
such that 
\begin{equation}
|[b,T]f(x)-[b,T^{\delta}]f(x)|\leq C\delta\|\nabla b\|_{L^{\infty}}Mf(x)\qquad\text{a.e. for every }\delta>0.\label{eq:Lemma3.10}
\end{equation}
Consequently, for every $p>1$, 
\[
\lim_{\delta\rightarrow0}\|[b,T]f(x)-[b,T^{\delta}]f(x)\|_{L^{p}}=0.
\]
\end{lem}

\begin{proof}
Let $f\in L^{p}.$ Let 
\[
[b,T_{\varepsilon}]f(x)=\int_{\varepsilon<|x-y|}(b(x)-b(y))K(x,y)f(y)dy
\]
First we would like to note that 
\begin{equation}
[b,T]f(x)=\lim_{\varepsilon\rightarrow0^{+}}[b,T_{\varepsilon}]f(x)\qquad\text{a.e.}\label{eq:bTae}
\end{equation}
since the corresponding result holds for $T_{\varepsilon}$ and $T$.
Let
\[
T_{\varepsilon}^{\delta}f(x)=\int_{\varepsilon<|x-y|}K^{\delta}(x,y)f(y)dy
\]
Note that if $\varepsilon>2\delta$
\[
|T_{\varepsilon}^{\delta}f(x)-T_{\varepsilon}f(x)|=0.
\]
If $0<\varepsilon<2\delta$ then 
\begin{align*}
|T_{\varepsilon}^{\delta}f(x)-T_{\varepsilon}f(x)| & =\left|\int_{\varepsilon<|x-y|<2\delta}K(x,y)f(y)dy\right.\\
 & -\left.\int_{\delta<|x-y|<2\delta}K^{\delta}(x,y)f(y)dy\right|\\
 & \leq|I(x)|+|II(x)|
\end{align*}
For $I$ we note that
\begin{align*}
|I(x)| & \leq\left|\int_{\varepsilon<|x-y|}K(x,y)f(y)dy\right|+\left|\int_{2\delta<|x-y|}K(x,y)f(y)dy\right|\\
 & \leq2\sup_{\rho>0}|T_{\rho}(f)(x)|
\end{align*}
For $II$,
\begin{align*}
|II(x)| & \leq\int_{\delta<|x-y|<2\delta}|K^{\delta}(x,y)||f(y)|dy\\
 & \leq\int_{\delta<|x-y|<2\delta}\frac{|f(y)|}{|x-y|^{n}}dy\\
 & \lesssim\frac{1}{|B(x,2\delta)|}\int_{|x-y|<2\delta}|f(y)|dy\\
 & \leq Mf(x).
\end{align*}
This yields that 
\begin{align*}
(T^{\delta})^{*}f(x)=\sup_{\varepsilon>0}\left|T_{\varepsilon}^{\delta}f(x)\right| & \leq\sup_{\varepsilon>0}|T_{\varepsilon}^{\delta}f(x)-T_{\varepsilon}f(x)|+\sup_{\varepsilon>0}|T_{\varepsilon}f(x)|\\
 & \lesssim2\sup_{\rho>0}|T_{\rho}f(x)|+Mf(x)
\end{align*}
Note that then $(T^{\delta})^{*}$ is bounded on $L^{p}$, actually
$\{(T^{\delta})^{*}\}_{\delta>0}$ are uniformly bounded on $L^{p}$,
and consequently, taking into account Lemma \ref{lem:Trunc} and also
that for every $x$, the following limit exists
\[
\lim_{\varepsilon\rightarrow0^{+}}\int_{1>|x-y|>\varepsilon}K^{\delta}(x,y)dy,
\]
then for every $f\in L^{p}$, 
\[
T^{\delta}f(x)=\lim_{\varepsilon\rightarrow0}\int_{|x-y|>\varepsilon}K^{\delta}(x,y)f(y)dy\qquad\text{a.e.}
\]
This fact allows us to say that if $f\in L^{p}$, then
\begin{equation}
[b,T^{\delta}]f(x)=\lim_{\varepsilon\rightarrow0^{+}}\int_{|x-y|>\varepsilon}(b(x)-b(y))K^{\delta}(x,y)f(y)dy\qquad\text{a.e.}\label{eq:limaebTdelta}
\end{equation}
Now we are in the position to prove \ref{eq:Lemma3.10}. For every
$x\in\mathbb{R}^{n}$ we have that if $\varepsilon<\delta,$
\begin{align*}
|[b,T_{\varepsilon}]f(x)-[b,T_{\varepsilon}^{\delta}]f(x)| & =\left|\int_{\varepsilon<|x-y|<2\delta}(b(x)-b(y))K(x,y)f(y)dy\right.\\
 & \left.-\int_{\delta<|x-y|<2\delta}(b(x)-b(y))K^{\delta}(x,y)f(y)dy\right|\\
 & \leq|I_{1}(x)|+|I_{2}(x)|.
\end{align*}
By the smoothness of $b$ and the size condition of $K$, 
\begin{align*}
|I_{1}(x)| & \leq\int_{\varepsilon<|x-y|<\delta}|b(x)-b(y)|\frac{1}{|x-y|^{n}}|f(y)|dy\\
 & \leq\|\nabla b\|_{L^{\infty}}\sum_{j=0}^{\infty}\int_{\frac{2\delta}{2^{j+1}}<|x-y|<\frac{2\delta}{2^{j}}}\frac{|f(y)|}{|x-y|^{n-1}}dy\\
 & \lesssim\|\nabla b\|_{L^{\infty}}\sum_{j=0}^{\infty}\frac{2\delta}{2^{j+1}}\frac{|B(0,1)|}{\left(\frac{2\delta}{2^{j+1}}\right)^{n}|B(0,1)|}\int_{\frac{2\delta}{2^{j+1}}<|x-y|<\frac{2\delta}{2^{j}}}|f(y)|dy\\
 & \lesssim\|\nabla b\|_{L^{\infty}}\sum_{j=0}^{\infty}\frac{2\delta}{2^{j+1}}Mf(x)\lesssim\|\nabla b\|_{L^{\infty}}\delta Mf(x).
\end{align*}
For $I_{2}(x)$, we have that
\begin{align*}
|I_{2}(x)| & \leq\int_{\delta<|x-y|<2\delta}|b(x)-b(y)||K^{\delta}(x,y)||f(y)|dy\\
 & \leq\|\nabla b\|_{L^{\infty}}\int_{\delta<|x-y|<2\delta}\frac{|f(y)|}{|x-y|^{n-1}}dy\\
 & \lesssim\|\nabla b\|_{L^{\infty}}\delta\frac{1}{|B(x,\delta)|}\int_{|x-y|<\delta}|f(y)|dy\\
 & \leq\|\nabla b\|_{L^{\infty}}\delta Mf(x).
\end{align*}
Then 
\[
|[b,T_{\varepsilon}]f(x)-[b,T_{\varepsilon}^{\delta}]f(x)|\lesssim\|\nabla b\|_{L^{\infty}}\delta Mf(x)
\]
taking into account (\ref{eq:bTae}) and (\ref{eq:limaebTdelta}),
letting $\varepsilon\rightarrow0$ we have that the desired inequality
holds. The remainder of the result follows from the $L^{p}$ boundedness
of $M$ just letting $\delta\rightarrow0$.
\end{proof}
From the result above it follows that to conclude the proof it suffices
to show that if we fix $\delta>0$ and $b\in C_{c}^{1}(\mathbb{R}^{n})$,
then $[b,T^{\delta}]$ is compact on $L^{p}$ with $p>r'$. In order
to do that we are going to show that the conditions in Lemma \ref{lem:Yosida}
hold.

First we note that $H=\{[b,T^{\delta}]f\ :\ \|f\|_{L^{p}}\leq1\}$
is bounded due to the boundedness of $[b,T^{\delta}]$ on $L^{p}$.
Note that this boundedness follows from the Lemma above. Indeed
\begin{align*}
\|[b,T^{\delta}]f\|_{L^{p}} & \leq\|[b,T^{\delta}]f-[b,T]f\|_{L^{p}}+\|[b,T]f\|_{L^{p}}\\
 & \leq c\|Mf\|_{L^{p}}+\|[b,T]f\|_{L^{p}}\lesssim\|f\|_{L^{p}}.
\end{align*}
We continue our proof showing that 
\begin{equation}
\lim_{h\rightarrow0}\|[b,T^{\delta}]f(\cdot)-[b,T^{\delta}]f(\cdot+h)\|_{L^{p}}=0\label{eq:C1}
\end{equation}
uniformly on $f\in H$. We begin writing
\begin{align}
 & [b,T^{\delta}]f(x)-[b,T^{\delta}]f(x+h)\nonumber \\
 & =\int_{\mathbb{R}^{n}}(b(x)-b(y))K^{\delta}(x,y)f(y)dy-\int_{\mathbb{R}^{n}}(b(x+h)-b(y))K^{\delta}(x+h,y)f(y)dy\nonumber \\
 & =(b(x)-b(x+h))\int_{\mathbb{R}^{n}}K^{\delta}(x,y)f(y)dy\nonumber \\
 & +\int_{\mathbb{R}^{n}}(b(x+h)-b(y))(K^{\delta}(x,y)-K^{\delta}(x+h,y))f(y)dy\nonumber \\
 & =\int_{\mathbb{R}^{n}}I_{1}(x,y,h)dy+\int_{\mathbb{R}^{n}}I_{2}(x,y,h)dy.\label{eq:Split}
\end{align}
We treat each term separatedly. We begin with $I_{1}$.
\begin{align*}
\left|\int I_{1}(x,y,h)dy\right| & \leq\|\nabla b\|_{L^{\infty}}|h|\left|\int K^{\delta}(x,y)f(y)dy\right|\\
 & =\|\nabla b\|_{L^{\infty}}|h|\left|\int_{|x-y|>\delta}K(x,y)f(y)dy\right|\\
 & \leq\|\nabla b\|_{L^{\infty}}|h|T^{*}f(x).
\end{align*}
and hence, since as we noted above by \cite{G,R} $T^{*}f$ is bounded
on $L^{p}$ we have that if $\|f\|_{L^{p}}\leq1$
\begin{align*}
\left\Vert \int_{\mathbb{R}^{n}}I_{1}(\cdot,y,h)dy\right\Vert _{L^{p}} & \leq\|\nabla b\|_{L^{\infty}}|h|\left(\int_{\mathbb{R}^{n}}\left|T^{*}f(x)\right|^{p}dx\right)^{\frac{1}{p}}\\
 & \lesssim\|\nabla b\|_{L^{\infty}}|h|\|f\|_{L^{p}}\leq\|\nabla b\|_{L^{\infty}}|h|
\end{align*}
and consequently
\begin{equation}
\left\Vert \int_{\mathbb{R}^{n}}I_{1}(\cdot,y,h)dy\right\Vert _{L^{p}}\rightarrow0\label{eq:Piece1}
\end{equation}
when $|h|\rightarrow0$ uniformly on $\|f\|_{L^{p}}\leq1$. Now we
deal with $I_{2}$. We split the integral of $I_{2}$ into four regions
\begin{align*}
R_{1} & =\left\{ y\in\mathbb{R}^{n}\ :\ |x-y|>\delta,\ |x+h-y|>\delta\right\} ,\\
R_{2} & =\left\{ y\in\mathbb{R}^{n}\ :\ |x-y|>\delta,\ |x+h-y|<\delta\right\} ,\\
R_{3} & =\left\{ y\in\mathbb{R}^{n}\ :\ |x-y|<\delta,\ |x+h-y|>\delta\right\} \\
R_{4} & =\left\{ y\in\mathbb{R}^{n}\ :\ |x-y|<\delta,\ |x+h-y|<\delta\right\} 
\end{align*}
Note that by the definition of $K^{\delta}$, clearly $\int_{R_{4}}I_{2}(x,y,h)dy=0.$
Then it suffices to study the integral over the remainder of $R_{i}$.
Let us begin with $R_{1}$. First note that if $|h|<\frac{\delta}{2}$,
then $x,x+h\in\frac{1}{2}B(x,\delta)$.Then
\begin{align*}
 & \left|\int_{R_{1}}I_{2}(x,y,h)dy\right|\\
 & \leq\int_{R_{1}}|(b(x+h)-b(y))(K^{\delta}(x,y)-K^{\delta}(x+h,y))f(y)|dy\\
 & \leq2\|b\|_{L^{\infty}}\int_{|x-y|>\delta}|K^{\delta}(x,y)-K^{\delta}(x+h,y)||f(y)|dy\\
 & =2\|b\|_{L^{\infty}}\sum_{m=1}^{\infty}\int_{2^{m}B(x,\delta)\setminus2^{m-1}B(x,\delta)}|K^{\delta}(x,y)-K^{\delta}(x+h,y)||f(y)|dy\\
 & =2\|b\|_{L^{\infty}}\sum_{m=1}^{\infty}\left|2^{m}B\right|\frac{1}{\left|2^{m}B\right|}\int_{2^{m}B(x,\delta)\setminus2^{m-1}B(x,\delta)}|K^{\delta}(x,y)-K^{\delta}(x+h,y)||f(y)|dy\\
 & \leq2\|b\|_{L^{\infty}}\sum_{m=1}^{\infty}\left[\left|2^{m}B\right|\left(\frac{1}{\left|2^{m}B\right|}\int_{2^{m}B(x,\delta)\setminus2^{m-1}B(x,\delta)}|K^{\delta}(x,y)-K^{\delta}(x+h,y)|^{r}dy\right)^{\frac{1}{r}}\right.\\
 & \times\left.\left(\frac{1}{\left|2^{m}B\right|}\int_{2^{m}B(x,\delta)}|f(y)|^{r'}dy\right)^{\frac{1}{r'}}\right]\\
 & \leq2\|b\|_{L^{\infty}}M_{r'}f(x)\sum_{m=1}^{\infty}\left|2^{m}B\right|\left(\frac{1}{\left|2^{m}B\right|}\int_{2^{m}B(x,\delta)\setminus2^{m-1}B(x,\delta)}|K^{\delta}(x,y)-K^{\delta}(x+h,y)|^{r}dy\right)^{\frac{1}{r}}\\
 & \leq2\|b\|_{L^{\infty}}M_{r'}f(x)\sum_{m=0}^{\infty}\left|2^{m}B\right|^{\frac{1}{r'}}\max\left\{ \frac{|h|^{\gamma-\frac{n}{r'}}}{|B|^{\frac{\gamma}{n}}}2^{-m\gamma},\frac{|h|}{|B|^{(1+\frac{n}{r'})\frac{1}{n}}}2^{-(1+\frac{n}{r'})m}\right\} \\
 & \leq2c\|b\|_{L^{\infty}}M_{r'}f(x)\max\left\{ |h|^{\gamma-\frac{n}{r'}},|h|\right\} \left(\sum_{m=0}^{\infty}\frac{\left|2^{m}B\right|^{\frac{1}{r'}}}{|B|^{\frac{\gamma}{n}}}2^{-m\gamma}+\sum_{m=0}^{\infty}\frac{\left|2^{m}B\right|^{\frac{1}{r'}}}{|B|^{(1+\frac{n}{r'})\frac{1}{n}}}2^{-(1+\frac{n}{r'})m}\right)\\
 & \lesssim\|b\|_{L^{\infty}}\max\left\{ |h|^{\gamma-\frac{n}{r'}},|h|\right\} M_{r'}f(x).
\end{align*}
Consequently if $\|f\|_{L^{p}}\leq1$, 
\[
\left\Vert \int_{R_{1}}I_{2}(\cdot,y,h)dy\right\Vert _{L^{p}}\leq2c_{\delta}\|b\|_{L^{\infty}}\max\left\{ |h|^{\gamma-\frac{n}{r'}},|h|\right\} \|M_{r'}f\|_{L^{p}}\lesssim\max\left\{ |h|^{\gamma-\frac{n}{r'}},|h|\right\} 
\]
and hence 
\begin{equation}
\left\Vert \int_{R_{1}}I_{2}(\cdot,y,h)dy\right\Vert _{L^{p}}\rightarrow0\label{eq:Piece21}
\end{equation}
when $|h|\rightarrow0$ uniformly on $\|f\|_{L^{p}}\leq1$.

The integrals over $R_{2}$ and $R_{3}$ are symmetric so we deal
just with $R_{2}$. We may assume that $|h|$ is very small. Let $\rho>\delta+|h|$
such that $b$ vanishes outside the ball $B_{0}=B(0,\rho)$. Note
that then $b(\cdot+h)$ has support in $2B_{0}$. Then, since $R_{2}\subset B(x,|h|+\delta)\subset B(0,\rho)$,
we have for $|x|<3\rho$ that $R_{2}\subset4B_{0}$. Hence, using
the size condition of $K^{\delta}$, 
\begin{align*}
 & \left|\int_{R_{2}}I_{2}(x,y,h)dy\right|\\
\leq & \left|\int_{R_{2}}(b(x+h)-b(y))(K^{\delta}(x,y)-K^{\delta}(x+h,y))f(y)dy\right|\\
= & \left|\int_{R_{2}}(b(x+h)-b(y))K(x,y)\varphi(x-y)f(y)dy\right|\\
= & \left|\int_{R_{2}}(b(x+h)-b(y))K(x,y)\varphi(x-y)f(y)dy\right|\\
\leq & \int_{4B_{0}\cap R_{2}}\left|(b(x+h)-b(y))K(x,y)\varphi(x-y)f(y)\right|dy\\
\lesssim & \|b\|_{L^{\infty}}\int_{R_{2}\cap4B_{0}}\frac{|f(y)|}{|x-y|^{n}}dy\\
\lesssim & \frac{1}{\delta^{n}}\|b\|_{L^{\infty}}\frac{|4B_{0}|}{|4B_{0}|}\int_{R_{2}\cap4B_{0}}|f(y)|dy\\
\leq & \frac{1}{\delta^{n}}\|b\|_{L^{\infty}}|4B_{0}|\left(\frac{1}{|4B_{0}|}\int_{4B_{0}}|f(y)|^{r'}dy\right)^{\frac{1}{r'}}\left(\frac{|R_{2}|}{|4B_{0}|}\right)^{\frac{1}{r}}\\
\lesssim & \frac{\rho^{\frac{n}{r'}}}{\delta^{n}}\|b\|_{L^{\infty}}|R_{2}|^{\frac{1}{r}}M_{r'}f(x)
\end{align*}
In the case $|x|\geq3\rho$, since $|h|<\rho$ and $|x+h|>2\rho$,
we have that $b(x+h)=0$, and hence, arguing as above 
\begin{align*}
 & \left|\int_{R_{2}}I_{2}(x,y,h)dy\right|\\
 & \leq\int_{4B_{0}\cap R_{2}}\left|(b(x+h)-b(y))K(x,y)\varphi(x-y)f(y)\right|dy\\
 & =\int_{4B_{0}\cap R_{2}}\left|b(y)K(x,y)\varphi(x-y)f(y)\right|dy\\
 & \lesssim\|b\|_{L^{\infty}}\int_{R_{2}\cap4B_{0}}\frac{|f(y)|}{|x-y|^{n}}dy\\
 & \lesssim\frac{\rho^{\frac{n}{r'}}}{\delta^{n}}\|b\|_{L^{\infty}}|R_{2}|^{\frac{1}{r}}M_{r'}f(x)
\end{align*}
Hence, gathering the estimates above, for $\|f\|_{L^{p}}\leq1$, 
\begin{align*}
\left\Vert \int_{R_{2}}I_{2}(\cdot,y,h)dy\right\Vert _{L^{p}} & \lesssim\frac{\rho^{\frac{n}{r'}}}{\delta^{n}}\|b\|_{L^{\infty}}\|M_{r'}f\|_{L^{p}}|R_{2}|^{\frac{1}{r}}\\
 & \lesssim\frac{\rho^{\frac{n}{r'}}}{\delta^{n}}\|b\|_{L^{\infty}}|R_{2}|^{\frac{1}{r}}
\end{align*}
Since $|R_{2}|\rightarrow0$ as $|h|\rightarrow0$ then 
\begin{equation}
\left\Vert \int_{R_{2}}I_{2}(\cdot,y,h)dy\right\Vert _{L^{p}}\rightarrow0\label{eq:Piece22}
\end{equation}
when $|h|\rightarrow0$ uniformly on $\|f\|_{L^{p}}\leq1$. As we
mentioned above, a similar argument yields
\begin{equation}
\left\Vert \int_{R_{3}}I_{2}(\cdot,y,h)dy\right\Vert _{L^{p}}\rightarrow0\label{eq:Piece23}
\end{equation}
uniformly on $\|f\|_{L^{p}}\leq1$ when $|h|\rightarrow0$. Consequently,
(\ref{eq:C1}) follows combining (\ref{eq:Split}), (\ref{eq:Piece1}),
(\ref{eq:Piece21}), (\ref{eq:Piece22}) and (\ref{eq:Piece23}).

Finally we study the decay at infinity. Since $b$ is compactly supported,
there exists $R_{0}>0$ such that $\supp b\subset B(0,R_{0})$. Let
$x$ such that $|x|>R>3R_{0}$. Then $x\not\in\supp b$ so
\begin{align*}
|[b,T^{\delta}]f(x)| & =|b(x)T^{\delta}f(x)-T^{\delta}(fb)(x)|=\left|T^{\delta}(fb)(x)\right|\\
 & =\left|\int_{\supp b}b(y)K^{\delta}(x,y)f(y)dy\right|\\
 & \lesssim\|b\|_{L^{\infty}}\int_{\supp b}\frac{|f(y)|}{|x-y|^{n}}dy\\
 & \lesssim\|b\|_{L^{\infty}}\frac{1}{|x|^{n}}\int_{B(0,R_{0})}|f(y)|dy\\
 & \leq\|b\|_{L^{\infty}}\frac{1}{|x|^{n}}|B(0,R_{0})|^{\frac{1}{p'}}\|f\|_{L^{p}}.
\end{align*}
Hence for $R>3R_{0}$, 
\begin{align*}
\|\chi_{\{|x|>R\}}[b,T^{\delta}]f\|_{L^{p}} & \lesssim\left\Vert \frac{1}{|\cdot|^{n}}\chi_{\{|x|>R\}}(\cdot)\right\Vert _{L^{p}}\|b\|_{L^{\infty}}|B(0,R_{0})|^{\frac{1}{p'}}\|f\|_{L^{p}}\\
 & \simeq\frac{1}{R^{n-\frac{n}{p}}}\|b\|_{L^{\infty}}|B(0,R_{0})|^{\frac{1}{p'}}\|f\|_{L^{p}}
\end{align*}
and letting $R\rightarrow\infty$ leads to the desired conclusion.

\subsubsection{Proof of Theorem \ref{thm:compLrHor}}

Let $T$ be a $L^{r}$-Hörmander operator. We begin observing $L^{r}$-Hörmander
condition presented here was borrowed from \cite{BCADH} and is implied
by the one in papers such as \cite{LMRdT,LMRRR,LRdT,LRdT1,IFRR}.
In those papers it was shown that if $b\in BMO$ then $[b,T]$ is
bounded on $L^{p}(w)$ for every $w\in A_{p/r'}$ provided that $p>r'$
($w\in A_{p/r'}^{+}$ in the one sided setting). By inspection of
the proofs it can be easily checked that the dependence on the boundedness
constant on the $A_{p/r'}$ ($A_{p/r'}^{+}$ in the one-sided setting)
constant is increasing. In particular if we fix $p_{0}>r'$ we have
that $[b,T]$ is bounded on $L^{p_{0}}(w)$ for every weight $w\in A_{p_{0}/r'}$
($w\in A_{p/r'}^{+}$ in the one sided setting). Consequently, the
combination of this fact, Theorem \ref{thm:compLrHorUnweighted} and
the main result of \cite{HL} or \ref{thm:CorUnWeighted} in the one
sided setting yields the desired conclusion.

\subsubsection{Proof of Theorem \ref{thm:CompactFracUnw}}

It suffices to exploit results in \cite{GWY}. From the results settled
there it follows the following particular case. Let $\alpha\in(0,1)$
and 
\[
T_{K_{\alpha}}f(x)=\int_{\mathbb{R}}K_{\alpha}(x,y)f(y)dy\qquad x\not\in\supp f
\]
with kernel $K_{\alpha}$ such that 
\begin{equation}
|K_{\alpha}(x,y)|\leq\frac{1}{|x-y|^{1-\alpha}}\label{eq:size}
\end{equation}
and for $|x-y|>2|x-x'|$
\begin{equation}
|K_{\alpha}(x,y)-K_{\alpha}(x',y)|+|K_{\alpha}(y,x)-K_{\alpha}(y,x)|\leq\rho\left(\frac{|x-x'|}{|x-y|}\right)\frac{1}{|x-y|^{1-\alpha}}\label{eq:smooth}
\end{equation}
where $\rho:[0,1]\rightarrow[0,\infty)$ is a modulus of continuity,
namely, $\rho$ is continuous, increasing, subadditive, $\rho(0)=0$
and 
\[
\int_{0}^{1}\rho(t)\frac{dt}{t}<\infty.
\]
For that class of operators we have the following result which is
a particular case of \cite[Theorem 1.5]{GWY}.
\begin{thm}
\label{thm:GuoWuYang}Let $1<p<q<\infty$, $0\leq\alpha<1$, $\alpha=\frac{1}{p}-\frac{1}{q}.$
Then if $b\in CMO(\mathbb{R})$ then $[b,T_{K_{\alpha}}]$ is compact
from $L^{p}$ to $L^{q}$.
\end{thm}

It is straightforward to check that Theorem \ref{thm:CompactFracUnw}
is a direct Corollary of this result, just choosing
\[
K_{\alpha}(x,y)=\frac{1}{(y-x)^{1-\alpha}}\chi_{\left\{ z>0\right\} }(y-x).
\]
Details are left to the reader. 

\subsubsection{Proof of Theorem \ref{thm:compFrac}}

It was shown in \cite{BL} that if $1<p<q<\infty$ with $\frac{1}{p}-\frac{1}{q}=\alpha$
then $[b,I_{\alpha}^{+}]:L^{p}(w^{p})\rightarrow L^{q}(w^{q})$ for
every $w\in A_{p,q}^{+}$. By inspection of the proofs it can be easily
checked that the dependence on the boundedness constant on the $A_{p,q}^{+}$
constant is increasing. In particular if we fix $1<p_{0}<q_{0}<\infty$
such that $\frac{1}{p_{0}}-\frac{1}{q_{0}}=\alpha$ we have that $[b,I_{\alpha}^{+}]:L^{p}(w^{p})\rightarrow L^{q}(w^{q})$
for every weight $w\in A_{p_{0},q_{0}}^{+}$. The combination of this
fact, Theorem \ref{thm:CompactFracUnw} and \ref{thm:CorUnWeightedApq}
yields the desired conclusion.

\bibliographystyle{abbrv}
\bibliography{refes}

\end{document}